\documentclass[10pt, reqno]{amsart}
\usepackage{amssymb,amsfonts}
\usepackage{enumerate}
\usepackage{mathrsfs}
\usepackage{url}
\usepackage{tikz}
\usepackage{quiver}

\newcommand{\nnn}{\mathbb N}  

\newcommand{\rrr}{\mathbb R}    
\newcommand{\sss}{\mathbb S}

\newcommand{\del}{\partial} 
\newcommand{\id}{\textbf{id}}

\newtheorem{theorem}{Theorem}[section]
\newtheorem{thm}{Theorem}[section]
\newtheorem{corollary}[thm]{Corollary}
\newtheorem{proposition}[thm]{Proposition}
\newtheorem{lemma}[thm]{Lemma}

\theoremstyle{definition}
\newtheorem{definition}[thm]{Definition}

\newtheorem{notation}[thm]{Notation}

\theoremstyle{remark}
\newtheorem{remark}[thm]{Remark}

\title{Sliced Wasserstein Geodesics and Equivalence of metrics}
\author{John Seale Hopper}
\address{University of California, Los Angeles}
\email{jshopper@math.ucla.edu}
\date{\today}
\thanks{This work was partially supported by the Air Force Office of Scientific Research under Award No. FA9550-18-1-0502. A poster presentation of this paper was given at FoCM 2026 with support of NSF award DMS-2609677.\\
This paper was produced entirely without the use of generative AI.}

\begin{document}
\begin{abstract}
    This paper will introduce a family of sliced Wasserstein geodesics which are not standard Wasserstein geodesics, objects yet to be discovered in the literature. These objects exhibit how the geometric structure of the sliced Wasserstein space differs from the Wasserstein space. Some of these geodesics will only be Hölder continuous with respect to the Wasserstein metric and thus will provide a direct proof that sliced Wasserstein and regular Wasserstein metrics are not bi-Lipschitz equivalent in the cases previous unknown in the literature. We will be able to conclude that they are only equivalent in one dimension. Furthermore, we prove new bounds on the Hölder exponents between the sliced Wasserstein and Wasserstein metrics when restricted to bounded sets. 
\end{abstract}

\maketitle


\section{Introduction and Definitions}

Rabin, Peyrém, Delon, and Berot in \cite{Rabin} pioneered a replacement metric for the Wasserstein metric for solving the barycenter and gradient flow problems. This so-called sliced Wasserstein metric is computationally more efficient than the standard Wasserstein metric. In \cite{Rabin} these sliced metrics were successfully used in various texture mixing problems. Sliced Wassserstien metrics in the proceeding years have enabled many researchers to produce high quality results with less burdensome computational costs. Similar metrics, like the max-sliced Wasserstein metric, have also been used to great success. These accomplishments sparked the need for more theoretical results about these new sliced metrics.

Since sliced Wasserstein metrics are primarily used as a computationally efficient replacement for the Wasserstein metric, it is natural to ask in what sense the two metrics are equivalent. The work of Bonnet in \cite{Bon} showed key inequalities between these metrics when they were on compact domains. Later, Bayraktar and Gaoyue in \cite{guo} showed that these metrics were topologically equivalent but, in a special case, they were not bi-Lipschitz equivalent. A similar result about topological equivalences was separately done in a later paper, \cite{Asygurantee}. As time passed more theoretical differences were discovered, Kitagawa and Taktasu in \cite{jun}, following a similar method as \cite{guo}, proved the metrics were not bi-Lipschitz equivalent in a much wider range of cases. The work of Bietert, Goldfeld, Sadhu and Kato in \cite{Statrobust} implies that the sliced and max-sliced Wasserstein metrics were not bi-Lipschitz equivalent. The methods of \cite{guo}, \cite{jun}, and \cite{Statrobust} all machinery concerning the convergence of empirical measures in these metrics. Since reducing to empirical measures is a crucial step in most applications of Wasserstein metrics these rates have been well studied. The work in this paper takes a different approach that is more elementary and will not require results about empirical measures. 

We can also ask how these metrics compare geometrically. The recent work \cite{park} performed a close study of the 2-sliced Wasserstein metric. They discovered that the space of probability measures with finite second moment is not a length space. They also showed that the flux of an absolutely continuous curve need not be a vector-valued measure. Working in parallel, \cite{jun} showed that the space of probability measures with finite $p$th moment equipped with the sliced Wasserstein metric and other sliced metrics do not form length spaces. Recall that the sliced Wasserstein metric began as a method to solve the barycenter problem which, in special cases, is the same as finding the (intrinsic) geodesic. Therefore, an understanding of geodesics in this space is foundational the the applications of this metric. This is further emphasized by the fact that much of the study of standard Wasserstein metrics also rely on the theory of geodesics (see, for example, \cite{grad}). The fact that the sliced Wasserstein metric does not have similar properties implies that the barycenter problem may have vasty different solutions that the Wasserstein barycenter problem, though examples are not well studied.

This paper sets out primarily to study these geometric differences between the sliced and standard Wasserstein metrics. In particular, until now the only known geodesics with respect to the sliced metric were translations and dilations which are also Wasserstein geodesics. This left the question: ``are there others and do they coincide with Wasserstein geodesics?'' We answer this by constructing a family of sliced Wasserstein geodesics which are not Wasserstein geodesics in $\rrr^d$ for $d$ at least $3$. This family will have a flux that is not given by a vector-valued measure showing this phenomenon occurs even in curves as nice as (extrinsic) geodesics (the example in \cite{park} was only absolutely continuous). 

These geodesics alone deserve recognition since they expose the pathological properties of sliced Wasserstein metrics. We will mention applications of these geodesics for questions of topological equivalence in section \ref{bilip}. In particular, we will complete the work of \cite{guo} and \cite{jun} in proving the Wasserstein and sliced Wasserstein metrics are only bi-Lipschitz equivalent in one dimension. We will also discuss to what degree this result can be extended to the intrinsic metric. As a second application, we will prove new upper-bounds on the Hölder exponents $r(p, q, d)$ such that \begin{equation}
    W_p(\mu, \nu)\lesssim (SW_{p, q}(\mu, \nu))^{r(p, q, d)}
\end{equation}
for all $\mu, \nu\in \mathcal{P}(B_1(0))$ where $B_1(0)$ is the unit ball in $\rrr^d$. We note that \cite{Bon} provides lower-bounds on such $r(p, q, d)$ when $p>1$ and, after the first version of this paper was put on ARXIV, \cite{other} proved the sharp exponent $r(1, 1, d)$ for all $d$.

\subsection{Outline of the paper}

This paper is organized as follows. In the rest of this section, we will recall some background definitions we use in this paper. In section \ref{1dgeo}, preform a close study on a family of Wasserstein geodesics on $\rrr$ whose computations will be used later in the paper. In section \ref{swgeo}, we construct our new family of sliced Wasserstein geodesics. In section \ref{apps}, we preform the computations for our applications for bi-Lipchitz equivalence and Hölder exponents. Finally, in the appendix, we provide some results about $\infty$ sliced Wasserstein distances which were previously not considered in the literature.

\subsection{Background definitions}
We will recall the following definitions of sliced Wasserstein and sliced Monge-Kantorovich metrics as introduced in \cite{jun}. These are related to the Radon transform. We will let $R^\theta(x) := x\cdot \theta$, let $\mathcal{P}_p$ represent the space of probability measures with finite $p$th moment, $\mathcal{P}_\infty$ represents probability measures with compact support, and let $f_\sharp\mu$ denote the push-forward of the measure $\mu$ by the map $f$. 
\begin{definition}
For $1\leq p\leq +\infty$, and $1\leq q\leq +\infty$ and $\mu, \nu\in \mathcal{P}_p(\rrr^d)$ we define the the $p$, $q$ sliced Wasserstein metric to be
\begin{equation}
    SW_{p, q}(\mu, \nu): = \left[\frac{1}{\mathcal{H}^{d-1}(\mathbb{S}^{d-1})}\int_{\mathbb{S}^{d-1}} (W_p(R_\sharp^\theta \mu, R_\sharp^\theta \nu))^qd\mathcal{H}^{d-1}(\theta) \right ]^{1/q}.
\end{equation}
Where when $q=+\infty$ this is understood as $\sup_{\theta\in \mathbb{S}^{d-1}} W_p(R_\sharp^\theta \mu, R_\sharp^\theta \nu)$. Here $W_p$ is the $p$th Wasserstein distance on $\mathcal{P}_p(\rrr)$.
\end{definition}
    
\begin{remark}
Often in the literature $SW_{p, p}$ is called to as the ``$p$-th sliced Wasserstein'' metric and $SW_{p, \infty}$ is called the ``$p$-th max-sliced Wasserstein''  metric. 
\end{remark}
Until this paper, the case when $p=\infty$ was not considered, we give the appropriate constructions in the appendix. These constructions are largely follow previously published work except for the fact that $W_{\infty, q}$ and $W_\infty$ are not topologically equivalent.
\begin{notation}
    In a metric space $(X, d)$ we will say that a continuous path $[a, b]\ni t\mapsto x(t)$ is a geodesic if $d(x(t), x(s)) = |t-s|d(x(a), x(b))$. We say that a continuous path is rectifiable if $\sup_{n\in \nnn: t_0=a< t_1< \cdots< t_n=b}\sum_{i=1}^nd(x(t_{i-1}), x(t_i))<+\infty$, in which case we define this quantity to be its length $\textup{len}_d(x(\cdot))$. We say that a path is an intrinsic geodesic if it is a rectifiable curve with shortest length. That is, for all other rectifiable curves $y: [a, b]\rightarrow X$ with $y(a)=x(a)$ and $y(b)=x(b)$ we have that $\textup{len}_d(x(\cdot)) \leq \textup{len}_d(y(\cdot)).$ Every geodesic is an intrinsic geodesic but the converse is not true. Sometimes, for emphasis, we will call a geodesic an extrinsic geodesic. 
\end{notation}

\section{One-Dimensional Geodesics} \label{1dgeo}

It was shown in \cite{jun} that for $1< p$ and $1<q<+\infty$ a necessary and sufficient condition for a family of probability measures $\mu_t$ to define a $p$, $q$ sliced Wasserstein geodesic is that for every $\theta\in \mathbb{S}^{d-1}$ that $R_\sharp^\theta \mu_t$ be a geodesic with respect to the $W_p$ metric on $\mathcal{P}_p(\rrr)$. As such, in this section, we will discuss Wasserstein geodesics in $\mathbb{R}$ that will be relevant for the geodesics we construct in section \ref{swgeo}. In particular, we will look at geodesics between uniform measure on $[-1, 1]$ and a convex combination of this measure with a Dirac mass on its support.

In particular, for $0<\alpha< 1$ and $-1\leq \beta\leq 1$ a geodesic between 
\begin{equation}
    \mu^{\alpha, \beta}_0 =\frac{1}{2}\mathcal{L}|_{[-1, 1]} \hspace{1em} \textup{ and }\hspace{1em}\mu^{\alpha, \beta}_1 =\frac{1-\alpha}{2}\mathcal{L}|_{[-1, 1]} + \alpha \delta_\beta
\end{equation} is given by 

\begin{align}\label{1dgeo}
    \mu^{\alpha, \beta}_t &= \frac{(1-\alpha) \mathcal{L}|_{[-1, 1]\setminus B_{\alpha(1-t)}[\beta(1-\alpha(1-t))]}}{2(1-\alpha(1-t))} + \frac{\mathcal{L}|_{B_{\alpha(1-t)}[\beta(1-\alpha(1-t))]}}{2(1-t)} \\
    &=\frac{(1-\alpha)\mathcal{L}|_{[-1, 1]}}{2(1 - \alpha(1-t))}  + \frac{t\alpha\mathcal{L}|_{B_{\alpha(1-t)}[\beta(1-\alpha(1-t))]}}{2\alpha(1-t)(1 - \alpha(1-t))}\\
    \label{1dforproj}&=\frac{\mathcal{L}|_{[-1, 1]}}{2(1 - t+t(1-\alpha)^{-1})}  + \frac{(-t+t(1-\alpha)^{-1})\mathcal{L}|_{B_{\alpha(1-t)}[\beta(1-\alpha(1-t))]}}{2\alpha(1-t)(1-t+t(1-\alpha)^{-1})}
\end{align}


Where $B_r[x]$ is the open ball of radius $r$ around $x$ and $\mathcal{L}|_I$ is Lebesgue measure restricted to the set $I$. See Figure \ref{fig:fig1} for a visualization as, although the notation is cumbersome, the images are relatively simple. Equation (\ref{1dgeo}) will be used to show these are indeed one-dimensional geodesics and equation (\ref{1dforproj}) will be more natural when we wish to write these as projections of measures. To see these are one-dimensional geodesics we will use the following standard facts and notation about one-dimensional transport problems.

\begin{notation}
   (1) The cumulative distribution function (cdf) is $F_\mu(t) := \mu((-\infty, t))$.\\
    (2) The ``generalized inverse'' of the cdf is $F_\mu^\circ(s) : = \sup\{x\in \rrr: F_\mu(x)\leq s\}$ for $s\in [0, 1]$.

\end{notation}
\begin{figure}[h]
    \centering
    \includegraphics[width=120mm]{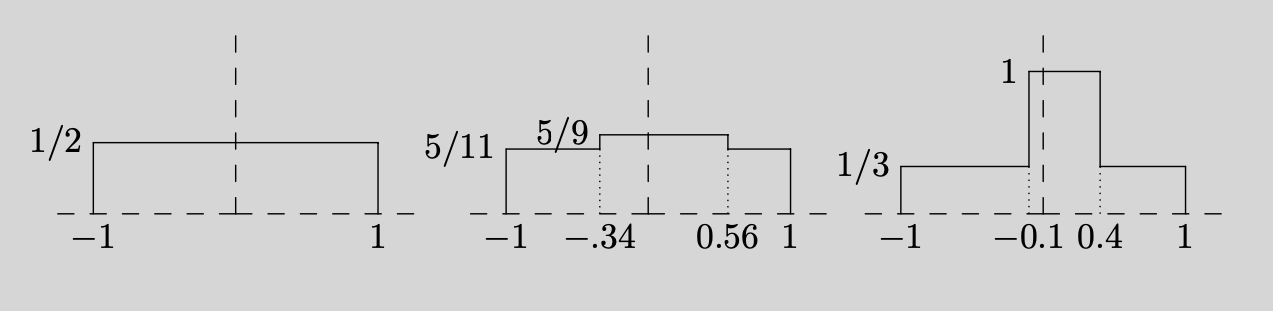}
    \caption{The probability distribution functions (pdf) for $\mu_{0}^{0.5, 0.2}$, $\mu_{.1}^{0.5, 0.2}$, and $\mu_{.5}^{0.5, 0.2}$ respectively.}
    \label{fig:fig1}
\end{figure}

\begin{theorem}\label{geocon} (Theorem 6.0.2 in \cite{grad}) Let $\mu$, $\nu$ be in $\mathcal{P}_p(\rrr)$ and let $c(x, y) = |x-y|^p$, which is convex, non-negative, and has p-growth.
If $\mu$ has no atoms, then $F_\nu^\circ \circ F_\mu$ is an optimal transport map from $\mu$ to $\nu$ for $1\leq p< +\infty$, and it is unique when $1<p<+\infty$.
    
\end{theorem} 

If $\mu^{\alpha, \beta}_0=\mu = \frac{1}{2}\mathcal{L}_{[-1, 1]}$ and $\nu =\mu^{\alpha, \beta}_1= \frac{\alpha}{2}\mathcal{L}_{[-1, 1]} + (1-\alpha)\delta_\beta $, then $\mu$ has no atoms and 
\begin{equation}\label{Fmu}
    F_\mu(s) = \begin{cases}
0  & s<-1\\
\frac{1}{2}(s+1) & -1\leq s\leq 1\\
1 & s>1
\end{cases},
\end{equation}
\begin{equation*}
    F_\nu(s) = \begin{cases}
0& s<-1\\
\frac{1-\alpha}{2}(s+1) & -1\leq s<\beta\\
\frac{1-\alpha}{2}(s+1) +\alpha & \beta\leq s\leq 1\\
1 & s>1
\end{cases}
\end{equation*}
thus, 
\begin{equation}\label{Fnuinv}
    F^\circ_\nu(s) = \begin{cases}
    -1+ \frac{2}{1-\alpha} s & s< \frac{(1+\beta)(1-\alpha)}{2}\\
    \beta & \frac{(1+\beta)(1-\alpha)}{2}\leq s \leq \frac{(1+\beta)(1-\alpha)}{2}+\alpha\\
    -1 + \frac{2}{1-\alpha}(s-\alpha) & s>\frac{(1+\beta)(1-\alpha)}{2}+\alpha
\end{cases}.
\end{equation}
According to \ref{geocon}, we can combine (\ref{Fmu}) and (\ref{Fnuinv}) to produce an optimal map 
\begin{equation}\label{transportmap}
    F^\circ_\nu\circ F_\mu (s) = \begin{cases} -1& s< -1\\
-1 +\frac{s+1}{1-\alpha} & -1\leq s<(1+\beta)(1-\alpha)-1\\
\beta & (1+\beta)(1-\alpha)-1\leq s< (1+\beta)(1-\alpha)-1+2\alpha\\
-1 +\frac{s+1-2\alpha}{1-\alpha} &(1+\beta)(1-\alpha)-1+2\alpha\leq s\leq 1\\
1 & s>1
\end{cases}.
\end{equation}
We can then check that 
\begin{equation}\label{interprotatemap}
    \lambda F^\circ_\nu\circ F_\mu  + (1-\lambda)\id = \begin{cases}
-\lambda +(1-\lambda)s& s< -1\\
\lambda\frac{\alpha}{1-\alpha} + s(1+\lambda(\frac{1}{1-\alpha}-1)) & -1\leq s<\beta -\alpha(1+\beta)\\
\lambda\beta + (1-\lambda)s & \beta -\alpha(1+\beta)\leq s< \beta +\alpha(1-\beta)\\
\lambda\frac{-\alpha}{1-\alpha} + s(1+\lambda(\frac{1}{1-\alpha}-1)) &\beta +\alpha(1-\beta)\leq s\leq 1\\
\lambda +(1-\lambda)s & s>1
\end{cases}.
\end{equation}

Where $\id$ is the identity map. Note that we simplified the conditions on $s$ and also collected the terms in $s$ from (\ref{transportmap}) to (\ref{interprotatemap}) so it is easier to conclude that the push-forward of uniform measure on $[-1, 1]$ (i.e. $\mu_0^{\alpha, \beta}$) is, 
\begin{equation}\label{1dpushforwardmap}
    (\lambda F^\circ_\nu\circ F_\mu  + (1-\lambda)\id)_\sharp \mu^{\alpha, \beta}_0 = \mu^{\alpha, \beta}_\lambda
\end{equation} (replace the $\lambda$ with a $t$ and take the reciprocal of the term in front of $s$. The factor of $1/2$ comes from the fact we are pushing forward Lebesgue measure weighed by $1/2$). We can check locations of the jumps in the pdf of $\mu_\lambda^{\alpha, \beta}$ by looking at the images of the points where the piece-wise definition changes which are $\lambda\beta + (1-\lambda)(\beta-\alpha(1+\beta)) = \beta -\alpha(1-\lambda)(1+\beta)$ and $\lambda\beta +(1-\lambda)(\beta +\alpha(1-\beta)) = \beta +\alpha(1-\lambda)(1-\beta)$ which are the boundary points of $B_{\alpha(1-\lambda)}(\beta-\alpha\beta(1-\lambda))$.

Thus we have shown in (\ref{1dpushforwardmap}) that each $\mu^{\alpha, \beta}_t$ is given by the push-forward of a convex combination of the optimal transport map and the identity. We can see that these are indeed one-dimensional Wasserstein geodesics by Theorem 7.2.2 in \cite{grad}. Furthermore, we know they are constant-speed geodesics and so we can note that \begin{equation*}
    W_p(\mu^{\alpha, \beta}_t, \mu^{\alpha, \beta}_s) = |t-s|W_p(\mu^{\alpha, \beta}_0, \mu^{\alpha, \beta}_1) 
\end{equation*}which we can use to compute the transport cost:

\begin{align*}
    W_p^p(&\mu^{\alpha, \beta}_0, \mu^{\alpha, \beta}_1) \\ &=\int_{-1}^{\beta-\alpha(1-\beta)} \Big|\frac{\alpha(1+s)}{1-\alpha}\Big|^p \frac{ds}{2} + \int_{\beta-\alpha(1-\beta)}^{\beta+\alpha(1+\beta)} \Big|\beta-s\Big|^p\frac{ds}{2} + \int_{\beta+\alpha(1+\beta)}^1 \Big|\frac{\alpha(s-1)}{1-\alpha}\Big|^p \frac{ds}{2}\\
    &= \frac{1}{2(p+1)}\left[\frac{\alpha^p}{(1-\alpha)^p}\Big(|1+\beta-\alpha(1-\beta)|^{p+1}+|\beta+\alpha(1+\beta)-1|^{p+1}\Big)\right]\\
    &\hspace{16em}+\frac{1}{2(p+1)}\Big[(\alpha(1+\beta))^{p+1} + (\alpha(1-\beta))^{p+1}\Big].
\end{align*}

Note that when $\beta=0$ this formula simplifies greatly to $\frac{1}{p+1}\alpha^p$. Thus, we have 
\begin{equation}\label{1dcost<inf}
    W_p(\mu^{\alpha, 0}_t, \mu^{\alpha, 0}_s) = \frac{\alpha}{(p+1)^{1/p}}|t-s|
\end{equation} for $1\leq p<+\infty$. We can next use the fact that these measures are compactly supported and that $W_\infty (\mu, \nu) = \lim_{p\rightarrow +\infty} W_p(\mu, \nu)$ to see that 
\begin{equation}\label{1dcostinf}
    W_\infty(\mu^{\alpha, 0}_t,\mu^{\alpha, 0}_s) = \alpha|s-t|
\end{equation} and so it is geodesic for $1\leq p\leq +\infty$. We now have sufficiently computed the one-dimensional Wasserstein flows for use in the following sections.

\section{Sliced Wasserstein Geodesics} \label{swgeo}

The phenomena that we will be studying occurs naturally in three dimensions, but by considering an embedding of $\rrr^3$ in higher-dimensional spaces we get similar effects. For concreteness, fix orthonormal basis of $\rrr^d$: $e_1, e_2, \cdots e_d$. We will use the following notation for surface (i.e. two-dimensional) measure restricted to a spherical shell in a three-dimensional subspace, as well as some simple maps we will want to push-forward with.

\begin{notation}
    (1) For $d\geq 3$, $r>0$ and $x\in \rrr^d$ define $$\sigma_{r, x}(\cdot) := \frac{\mathcal{H}^2|_{\{ae_1 + be_2 +ce_3 +x: a^2+b^2+c^2=r^2\}}(\cdot)}{\mathcal{H}^2(\{ae_1 + be_2 +ce_3: a^2+b^2+c^2=r^2\})} \hspace{2em} \sigma_{0, x} := \delta_x.$$
    (2) For $a\in \rrr$, define the map $M^a: \rrr^d\rightarrow \rrr^d$ by $M^a(x) = ax$.\\
    (3) Fox $x\in \rrr^d$ define the map $A^x:\rrr^d\rightarrow\rrr^d$ by $A^x(y) = x+y$.
\end{notation}

Calculus and surfaces of revolution will confirm that 
\begin{equation}
    R^\theta_\sharp \sigma_{r, x} = \frac{1}{2rs(\theta)}\mathcal{L}|_{[\theta\cdot x-rs(\theta), \theta\cdot x+rs(\theta)]}, \hspace{2em} \delta_{\theta\cdot x} \textup{ if } s(\theta)=0
\end{equation} where 
$s(\theta) = \sqrt{\sum_{i=1}^3 (e_i\cdot \theta)^2}$, the norm of $\theta$ projected onto $\textup{span}\{e_1, e_2, e_3\}$. Fix $x\in \overline{B_1(0)}$. We will consider the following family of measures in $\rrr^d$, 
\begin{equation} \label{geodef}
    \begin{split}
\nu^{\alpha, x}_t = \frac{1}{(1-t+t(1-\alpha)^{-1})}\sigma_{1, x} &+ \frac{(-t+ t(1-\alpha)^{-1})}{(1-t+t(1-\alpha)^{-1})}\sigma_{\alpha(1-t), x(1-\alpha(1-t))},\\
\nu^{\alpha, x}_0 = \sigma_{1, x}, \hspace{1em}&\textup{ and }\hspace{1em}\nu^{\alpha, x}_1 = (1-\alpha)\sigma_{1, x} + \alpha \delta_x.        
    \end{split}
\end{equation}

\begin{figure}[h]
    \centering
    \includegraphics[width=30mm]{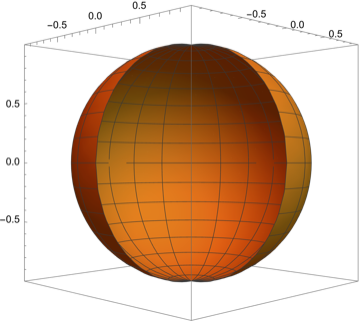}
    \includegraphics[width=30mm]{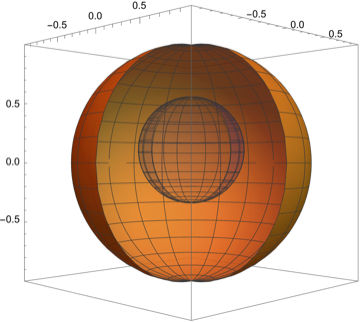}
    \includegraphics[width=30mm]{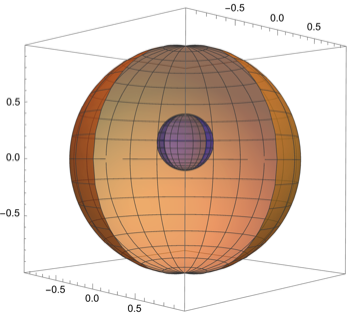}
    \includegraphics[width=30mm]{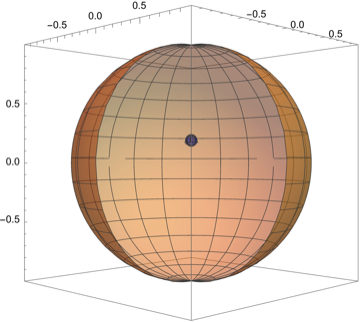}
    \caption{A visualization of $\nu^{0.5, (0, 0, .2)}_0$, $\nu^{0.5, (0, 0, .2)}_{.1}$, $\nu^{0.5, (0, 0, .2)}_{.5}$, and $\nu^{0.9, (0, 0, .2)}_{.5}$ respectively. Relative opacity represents the weighting versus $\sigma_{r, x}$. If we consider a push-forward by $R^{(0, 0, 1)}$ of the frist three, we recover the measures with pdfs depicted in figure \ref{fig:fig1}}.
    \label{fig:3dgeopic}
\end{figure}

In particular, we get that 
\begin{equation}
    R^\theta_\sharp \nu^{\alpha, x}_t = M^{s(\theta)}_\sharp\mu^{\alpha x\cdot \theta}_t.
\end{equation} That is $t\mapsto \nu^{\alpha, x}_t$ is a sliced Wasserstein (extrinsic) geodesic because each projection is a Wasserstein geodesic. Note, this uses the fact that if $\mu_t$ is a geodesic then so is $M^r_\sharp \mu_t$.

\begin{remark}\label{hop}
    Note that $\textup{spt}(\nu_t^{\alpha, 0})\subset \del B_1(0) \cup \del B_{\alpha(1-t)}(0)$ which is disconnected. The total mass in $\del B_1(0)$ is $1-\alpha t$ and the total mass in $B_{\alpha}(0)$ is $t\alpha$. The fact this is changing in time is an example of the fact that the tangent space for the sliced Wasserstein metric is different from the standard Wasserstein metric. For the standard Wasserstein metric, gradients correspond to linear interpolation at the particle level. If one attempted to follow the measures on the particle level, they would ``hop'' between these connected components and not move continuously. Note that \cite{park} shows this phenomena can occur in absolutely continuous curves in remark 3.9. The example given here is notable because it is the first example of this pathological ``hopping'' occurring with a curve as nice as a compactly supported geodesic.
\end{remark}

We should note that we can quickly grow this family of sliced Wasserstein geodesics by adding dilations and translations. It is a fact that if $t\mapsto \mu_t$ is a Wasserstein geodesic in $\rrr$, then $t\mapsto M^a_\sharp(\mu_t)$ and $t\mapsto A^{ty+z}_\sharp(\mu_t)$ are also geodesics for $a\in \rrr$ and $z, y\in \rrr$. As such, the initial restriction of being centered at the origin and having unit radius are simply to reduce notation. Thus, for all $\alpha\in (0, 1)$, $a\in \rrr_{>0}$, $x\in \overline{B_1(0)}$ and $y, z\in \rrr^d$, we can note that $t\mapsto M^a_\sharp(A^{ty +z}_\sharp (\nu^{\alpha, x}_t))$ are $SW_{p, q}$ geodesics. Note that the edge cases when $\alpha =0$ and $\alpha=1$ are also sliced Wasserstein geodesics, they are translations when $\alpha=0$ and pure dilation when $\alpha=1$, these geodesics when $\alpha\in \{0, 1\}$ where already known in the literature. Thus, we have found more than one novel sliced Wasserstein geodesic, we have actually found a five parameter family, showing that the dimension of the space of sliced Wasserstein geodesics which are not Wasserstein geodesics is at least $3d+2$ for any particular embedding of $\rrr^3$ into $\rrr^d$. 

Finally, we should add that we get a family of geodesics coming from measures in a two-dimensional subspace whose projections are uniform. We can define as follows
\begin{equation}
    \eta_{r, x}(A) := \iint_{(a, b)\in B_r(0)} \sqrt{4^2-a^2-b^2}\chi_{A}(ae_1+be_2+x) dadb 
\end{equation}

We give these geodesics less emphasis as they are not as useful for the applications in the following sections, but are of interest in finding as many sliced Wasserstein geodesics as possible.

\section{Applications}\label{apps}

\subsection{Application 1: Non Bi-Lipschitz equivalence of metrics}\label{bilip}

We will now use these geodesics to prove the following  version of Main Theorem (3) from \cite{jun} which was a strengthening of Theorem 2.1(iii) and a correction of Theorem 2.1(ii) in \cite{guo}. Note that there are some cases, particularly when $p=1$, that are shown in \cite{guo} and \cite{jun} that are not shown here, so this is not a strict strengthening, but we do cover all the remaining cases. This is the first paper - as far as the author is aware - to consider the case when $p=+\infty$ so, in addition to the cases left open by \cite{guo} and \cite{jun} we address this edge case. Attention should be drawn to the proof method, which is distinct from the previous papers. Instead of using a probabilistic method that only proved examples exist, this paper provides direct proof by constructing specific examples. Additionally, the examples given are all of uniformly bounded support, something done in \cite{jun} for only $p\geq 2$ and finite $q$, showing that bi-Lipschitz equivalence cannot be recovered by restricting to measures restricted to an arbitrary compact set.

\begin{theorem} \label{iff}
    For any $1\leq p, q\leq +\infty$, $(\mathcal{P}_p(\rrr^d), SW_{p, q})$ and $(\mathcal{P}_p(\rrr^d), W_{p})$ are bi-Lipschitz equivalent if and only if $d=1$. Moreover, when $p=+\infty$ they are topologically equivalent if and only if $d=1$.
\end{theorem}

\begin{lemma}\label{main}
    For all $1< p\leq +\infty$ and $1\leq q\leq +\infty$ and $d\geq 3$, $d\in \nnn$, the spaces $(\mathcal{P}_p(\rrr^d), W_p)$ and $(\mathcal{P}_p(\rrr^d), SW_{p, q})$ are not bi-Lipschitz equivalent, nor are they equivalent along geodesics in $(\mathcal{P}_p(\rrr^d), SW_{p, q})$. Moreover, for $1\leq q\leq +\infty$, $SW_{\infty, q}$ and $W_\infty$ are not topologically equivalent.
\end{lemma}

\begin{proof}[Proof of Lemma \ref{main}]

Consider the family of measures $\nu_{t}^{\alpha, 0}$ for $t>0$ and $\nu_0^{\alpha, 0}$. We have shown above that 
\begin{equation}\label{SWdistcomputation}
    \begin{split}
         SW_{p, q}(\nu_{t}^{\alpha, 0}, \nu_{0}^{\alpha, 0}) &= \left[\frac{1}{\mathcal{H}^{d-1}(\mathbb{S}^{d-1})}\int_{\mathbb{S}^{d-1}} \Big(W_p(R_\sharp^\theta \nu_{t}^{\alpha, 0}, R_\sharp^\theta \nu_0^{\alpha, 0})\Big)^qd\mathcal{H}^{d-1}(\theta) \right ]^{1/q}\\
    &= \left[\frac{1}{\mathcal{H}^{d-1}(\mathbb{S}^{d-1})}\int_{\mathbb{S}^{d-1}} \Big(W_p(M^{s(\theta)}_\sharp \mu_{t}^{\alpha, 0}, M^{s(\theta)}_\sharp \mu_0^{\alpha, 0})\Big)^qd\mathcal{H}^{d-1}(\theta) \right ]^{1/q}\\
    &= \left[\frac{1}{\mathcal{H}^{d-1}(\mathbb{S}^{d-1})}\int_{\mathbb{S}^{d-1}} \Big(s(\theta) t \frac{\alpha}{(1+p)^{1/p}}\Big)^qd\mathcal{H}^{d-1}(\theta) \right]^{1/q} \\
    &=  \frac{\alpha t}{(p+1)^{1/p}}\left[\frac{1}{\mathcal{H}^{d-1}(\mathbb{S}^{d-1})}\int_{\mathbb{S}^{d-1}} \Big(s(\theta)\Big)^qd\mathcal{H}^{d-1}(\theta) \right]^{1/q} \\
    &= \frac{\alpha t}{(p+1)^{1/p}} C_{d, q}.
    \end{split}
\end{equation}

We can note that $C_{d, q} = \left[\frac{1}{\mathcal{H}^{d-1}(\mathbb{S}^{d-1})}\int_{\mathbb{S}^{d-1}} \Big(s(\theta)\Big)^qd\mathcal{H}^{d-1}(\theta) \right]^{1/q}$ is a constant depending on $d$ and $q$. Note that $0\leq s(\theta)\leq 1$ and it is non-zero on a set of positive measure, so we can note that $0< C_{d, q} \leq 1$. It is clear when $q=+\infty$ that $C_{d, q} =1$.

On the other hand, we can compute $W_p(\nu_t^{\alpha, 0}, \nu_0^{\alpha, 0})$. The map $T: \rrr^d\rightarrow \rrr^d$ where $T(x) = \frac{x}{|x|}$ for $x\neq 0$ and $T(0)=0$ has c-monotone graph for all $c(x, y) = |x-y|^p$
$1\leq p< +\infty$ and so by Theorem 6.1.4 in \cite{grad} we know that this is an optimal transport map from $\nu_t^{\alpha, 0}$ to $T_\sharp(\nu_t^{\alpha, 0}) = \nu_0^{\alpha, 0}$ for $0<t<1$. Since it is optimal for all $1\leq p<+\infty$ we can, by taking limits, argue it is optimal for $p=+\infty$. We can then compute: 
\begin{equation}\label{Wdistgeo}
    \begin{split}
        W_p^p(\nu_t^{\alpha, 0}, \nu_0^{\alpha, 0}) &= \int |T(x)-x|^p d\nu_t^{\alpha, 0}(x) = \int_{B_{\alpha(1-t)}(0)} [1-\alpha(1-t)]^pd\nu_t^{\alpha, 0}\\ 
&=\frac{\alpha t}{(1-\alpha(1-t))}[1-\alpha(1-t)]^p = \alpha t (1-\alpha(1-t))^{p-1}.
    \end{split}
\end{equation}

We can then see that \begin{equation}\label{finitepratio}
    \frac{W_p(\nu_t^{\alpha, 0}, \nu_0^{\alpha, 0})}{SW_{p, q}(\nu_{t}^{\alpha, 0}, \nu_{0}^{\alpha, 0})}= \frac{(1-\alpha(1-t))^{1-1/p}(p+1)^{1/p}}{ C_{d, q}} (\alpha t)^{\frac{1}{p}-1}
\end{equation} when $1< p< +\infty$. Clearly for $p>1$ this expression goes to $+\infty$ as $t\rightarrow 0$ for all fixed $\alpha$ and so we have that $W_p$ and $SW_{p, q}$ are not bi-Lipschitz equivalent. 

The case when $p=+\infty$ we can note that \begin{equation}
    W_\infty(\nu_t^{\alpha, 0}, \nu_0^{\alpha, 0}) = 1-t(1-\alpha)
\end{equation} and \begin{equation}\label{eq:SWinf}
    SW_{\infty, q}(\nu_t^{\alpha, 0}, \nu_0^{\alpha, 0}) = \alpha t C_{d, q}.
\end{equation} Thus, we  can see that $\lim_{t\rightarrow 0} SW_{\infty, q}(\nu^{\alpha, 0}_t, \nu^{\alpha, 0}_0) =0$ and $\lim_{t\rightarrow 0} W_{\infty}(\nu^{\alpha, 0}_t, \nu^{\alpha, 0}_0) =1$ and so $\nu^{\alpha, 0}_t$ converges to $\nu^{\alpha, 0}_0$ with respect to $SW_{\infty, q}$ but not $W_\infty$ so they are not topologically equivalent
and so $SW_{\infty, q}$ and $W_\infty$ are not bi-Lipschitz equivalent either.

\end{proof}

Note that it is not known if the exponents here are optimal; we have here that $t\mapsto \nu_{t}^{\alpha, 0}$ are Lipschitz in $SW_{p, q}$ but only $\frac{1}{p}$-Hölder continuous with respect to $W_p$. One corollary of Theorem 5.1.5 in \cite{Bon} is that every compactly supported Lipschitz path (and so every compactly supported geodesic) in $SW_{p, q}$ in $\rrr^d$ is at least $\frac{1}{p(d+1)}$-Hölder continuous. Thus, it is open what the optimal exponent $\alpha$ is, such that every Lipschitz path and every compactly supported geodesic in $SW_{p, q}$ is $\alpha$-Hölder continuous with respect to $W_p$ or if the exponents for Lipschitz paths and geodesics are different (or change if we remove the requirement of being compactly supported).

The addendum about equivalence along geodesics in Lemma \ref{main} means that this theorem can also be applied to show facts about the intrinsic metric of $\mathcal{P}^p(\rrr^d)$. There was some hope (see, for example, \cite{cs}) in the community that the non-equivalence could be fixed by considering the intrinsic metric (the metric where the distance between two points is the length of the shortest rectifiable curve between them, see \cite{park} for more details). Unfortunately, because the intrinsic and extrinsic metrics agree along extrinsic geodesics, we get the immediate corollary.

\begin{definition} We define the $SW_{p, q}$ intrinsic metric, \begin{equation*}
    \ell_{SW_{p, q}} = \inf_{\{\mu_t\}_{t\in [0, 1]} \in AC(\mathcal{P}_p(\rrr^d), SW_{p, q})}\Bigg( \sup_{0=t_0< \dots < t_n=1; n\in \nnn}\sum_{i=0}^{n-1} SW_{p, q}(\mu_{t_i}, \mu_{t_{i+1}})\Bigg).
\end{equation*}
\end{definition}

\begin{corollary}
    If $d\geq 3$ and $1<p\leq +\infty$ and 
    $1\leq q\leq +\infty$, then $(\mathcal{P}_p(\rrr^d) ,\ell_{SW_{p, q}})$ and 
    $(\mathcal{P}_p(\rrr^d), W_p)$ are not bi-Lipschitz equivalent.
\end{corollary}

\begin{proof}

If $t\mapsto \nu_t^{\alpha, 0}$ is an extrinsic geodesic with respect to $SW_{p, q}$
\begin{equation*}
    SW_{p, q}(\nu_t^{\alpha, 0}, \nu_s^{\alpha, 0}) = |s-t|SW_{p, q}(\nu_0^{\alpha, 0}, \nu_1^{\alpha, 0}).
\end{equation*} Thus, the extrinsic and intrinsic metrics agree along this path: \begin{equation*}
    \ell_{SW_{p, q}}(\nu_t^{\alpha, 0}, \nu_s^{\alpha, 0}) = SW_{p, q}(\nu_t^{\alpha, 0}, \nu_s^{\alpha, 0}).
\end{equation*} The computations done in (\ref{finitepratio}) and (\ref{eq:SWinf}) are the same if we replace $SW_{p, q}$ with $\ell_{SW_{p, q}}$. Therefore, $(\mathcal{P}_p(\rrr^d) ,\ell_{SW_{p, q}})$ and 
    $(\mathcal{P}_p(\rrr^d), W_p)$ are not bi-Lipschitz equivalent.
    
\end{proof}

In the introduction, it was promised that this paper would fill the gaps left by \cite{jun} and \cite{guo}. As of yet this has not fully been done. There remains to show the cases where $d=2$, $1< p<+\infty$ and $q=+\infty$ as well as $p=\infty$ for any $q$. We will show these case separately. 

\begin{lemma} \label{extra}
    For $d=2$, $1< p\leq +\infty$, and $1\leq q\leq +\infty$,  $(\mathcal{P}_p(\rrr^d), W_p)$ and $(\mathcal{P}_p(\rrr^d), SW_{p, q})$ are not bi-Lipschitz equivalent. Moreover, for $1\leq q\leq +\infty$ $SW_{\infty, q}$ and $W_\infty$ are not topologically equivalent.
\end{lemma}

\begin{proof}
Beginning with the case where $1<p<+\infty$ and $1\leq q\leq +\infty$, we will consider the family of measures 
\begin{equation}\label{d2measure}
    \mu_t = t\delta_0 + (1-t)\frac{\mathcal{H}^1|_{\del(B_1(0))}}{2\pi}.
\end{equation}
We realize that the distribution function $R^\theta_\sharp \frac{\mathcal{H}^1|_{\del(B_1(0))}}{2\pi}$ is normalized arc-length,
$\frac{1}{\pi \sqrt{1-x^2}}$ call this measure $\nu$.

We can note that $\mu_t$ is spherically symmetric for all $t$ and so $SW_{p, q}(\mu_t, \mu_0) = W_p(\nu, t\delta_0+(1-t)\nu)$ which has transport map (due to symmetry we only consider $x\geq 0$) \begin{equation}\label{d2map}
    T(x) = \begin{cases}
0 & 0<x<\sin(\pi t/2)\\
\sin\Big(\frac{\arcsin(x)}{1-t} -\frac{\pi t}{2(1-t)}\Big) & \sin(\pi t/2) \leq x\leq 1
\end{cases}.
\end{equation}
Notice that $|T(x)-x|$ is maximized at $x= \sin(\pi t/2)$ where $|T(x)-x| = \sin(\pi t/2)\in O(t)$. As such, we can note that \begin{equation}\label{eq:d2SWdist}
    SW_{p, q} \leq \Big(\int_0^1 \sin(\pi t/2)^qdx \Big)^{1/q} \in O(t).
\end{equation}

On the other-hand, we can note that the optimal transport plan is given by \begin{equation*}
    \gamma_t =(1-t)\Delta_\sharp \mu+ t\mu \times \delta_0,
\end{equation*}
where $\Delta: \rrr^2\rightarrow \rrr^2\times \rrr^2$ is the diagonal map given by $\Delta(x) = (x, x)$. We can then see that \begin{equation}\label{eq:d2Wdist}
    \int |x-y|^p d\gamma(x, y) = \int_{\del B_1(0)\times \delta_0} 1^p d\gamma = t,
\end{equation}
so $W_p(\mu_t, \mu) = t^{1/p}$. For $1<p<+\infty$, we have that 
\begin{equation}
    \frac{SW_{p, q}(\mu_t, \mu)}{W_p(\mu_t, \mu)} \leq \frac{\sin(\pi t/2)}{t^{1/p}} \in O(t^{1-1/p}).
\end{equation} Thus, $\lim_{t\rightarrow 0}\frac{SW_{p, q}(\mu_t, \mu)}{W_p(\mu_t, \mu)} =0$, which is what we desired to show.

For $p =+\infty$ we will use the same family of measures, (\ref{d2measure}), and note that the same plans, (\ref{d2map}), are optimal as for $1<p<+\infty$. Note that $W_\infty(\mu, \mu_t) =1$ for all $0<t\leq 1$.
We can also find $W_\infty(R^\theta_\sharp \mu_t, R^\theta_\sharp \mu_0)$ to be the maximum displacement from above, which is $\frac{2\sin(t)}{\pi}$. We can then note that, due to the spherical symmetry of $\mu_t$, we know that \begin{equation*}
    SW_{\infty, q}(\mu_t, \mu_0) =W_\infty(R^\theta_\sharp \mu_t, R^\theta_\sharp \mu_0) =\frac{2\sin(t)}{\pi}.
\end{equation*} 
Thus, we  can see that $\lim_{t\rightarrow 0} SW_{\infty, q}(\mu_t, \mu) =0$ and $\lim_{t\rightarrow 0} W_{\infty}(\mu_t, \mu) =1$ and so $\mu_t$ converges to $\mu$ with respect to $SW_{\infty, q}$ but not $W_\infty$ so they are not topologically equivalent nor bi-Lipschitz equivalent.

\end{proof}

We have now covered all the previously unknown cases of bi-Lipschitz equivalence and now can prove Theorem \ref{iff} and characterize when these metrics are equivalent.

\begin{proof}[Proof of Theorem \ref{iff}]
It is trivially true that when $d=1$ $SW_{p, q} = W_p$ so they are equal and thus bi-Lipschitz equivalent. We can then recall that when $d>1$ they are not bi-Lipschitz equivalent: Lemma \ref{main} covers $d\geq 3$,  $1<p\leq +\infty$, with any $1\leq q\leq +\infty $; Lemma \ref{extra} covers $d=2$, $1< p\leq +\infty$, and all choices of $q$; Main Theorem (3) in \cite{jun}  handles $d\geq 2$, $p=1$, and any $1\leq q\leq +\infty$.

\end{proof}

We should note a shortcoming of these maps. Due to their symmetries they do not demonstrate the differences between the values of $q$ (which are partially open). As such, they are less applicable to differentiating sliced and max-sliced Wasserstein metrics as done in \cite{Statrobust}. Even so, the fact that Theorem \ref{iff} was previously done indirectly shows the need for more concrete examples to help researchers understand the nature of sliced Wasserstein metrics. 

\subsection{Application 2: Optimal Hölder exponents}

\begin{notation}
For $1\leq p, q\leq +\infty$ and $d\in \nnn$, let $I_{p, q, d}\subset [0, 1]$ be defined such that if $r\in I_{p, q, d}$ then there is a $C_{p, q, d, r}>0$ such that for all $\mu, \nu\in \mathcal{P}(\rrr^d)$ and $\textup{spt}(\mu), \textup{spt}(\nu) \subseteq \overline{B_1(0)}$ we have that
\[W(\mu, \nu)_p\leq C_{p, q, d, r}SW_{p, q} (\mu, \nu)^{r}\]
We then define $r(p, q, d) = \sup(I_{p, q, d})$.
\end{notation}

We can note that the $0\in I_{p, q, d}$ for all $1\leq p, q\leq +\infty$ and $d\in \nnn$ so the supremum in the definition$r(p, q, d)$ is not over an empty set. Given that $SW_{p, q}(\mu, \nu)\leq 2$ that if $r\in I_{p, q, d}$ with $C_{p, q, d, r}$ for all $0<\delta<r$ we have that $\delta\in I_{p, q, d}$ with $C_{p, q, d, \delta}\leq C_{p, q, d, r}2^{r-\delta}$. That is $I_{p, q, d}$ is either of the form $[0, r(p, q, d))$ or $[0, r(p,q ,d)]$.

We remark that the computations done in the earlier subsection are, in fact, computing upper-bounds on $r(p, q, d)$. Note that if $W_p(\mu, \nu)$ and $W_{p,q}(\mu, \nu)$ were bi-Lipschitz equivalent would imply $r(p, q, d)=1$. The computations in (\ref{SWdistcomputation}), (\ref{Wdistgeo}), (\ref{eq:d2SWdist}) and (\ref{eq:d2Wdist}) imply that $r(p, q, d)\leq 1/p$ for $d\geq 2$. We remark that a corollary of Theorem 5.1.5 in \cite{Bon} is that $r(p, q, d)\geq \frac{1}{p(d+1)}$ thus there is a gap which grows with dimension and as $p\rightarrow 1$. We can remark in the case that $p=q=1$ then in the recent work of \cite{other} it was shown that $r(1, 1, d) = \frac{1}{d}$. It was shown that $r(1, \infty, d)\geq \frac{2}{d+2}$ in \cite{bob}. As such, it is clear that the dependence on $q$ is important, at least for $p=1$ and likely for other values of $p$ even though it is not visible in the estimates proved here. 

We also remark that at this time most of the progress is restricted to the $p=1$ case whereas our estimates only give non-trivial bounds when $p>1$. It is of the view of the author that the $p=1$ and the $p>1$ cases are distinct as in the $p=1$ case the contractive nature of the projections $R_\theta$ is accounting for the difference in the two metrics. In the $p>1$ case, there is in addition to the fact that $R^\theta$ is a contraction, the fact that the push forward of the optimal transport map under $R^\theta$ is no longer monotone. For $p=1$ we know that the optimal map need not be monotone along transport rays so the fact, but this is not the case for $p>1$. One can view this as the major fact exploited in this paper, that fact that $R^\theta$ does not preserve the monotonicity of the transport maps in the examples of our geodesics and the contractive nature of $R^\theta$ was not fully exploited.

\textbf{Acknowledgments:}
I am greatly indebted to Wilfrid Gangbo who taught me optimal transport theory, guided me to sliced Wasserstein metrics, and advised me during the process of writing and research. I would also like to that Jun Kitagawa for fruitful discussion particularly about how to get results in higher dimensions. This work was partially supported by the Air Force Office of Scientific Research under Award No. FA9550-18-1-0502 and a poster presentation was given with the support of NSF award DMS-2609677.

\appendix
\section{The $SW_{\infty, q}$ metric}\label{append}

A number of the technical properties, like $(\mathcal{P}_p, SW_{p, q})$ being a complete separable metric space were only concretely written down in \cite{jun}. Since this paper is the first to consider the $p=+\infty$ case we have decided for the sake of completeness to prove some of the basic properties. There is nothing particularly novel about this, and so we have relegated it to the appendix.

\begin{notation}
    We say that $P_\infty(\rrr^d)$ is the space of probability measures that are compactly supported. 
\end{notation}  

We begin with a simple lemma, whose statement is almost identical to Lemma 2.2 in \cite{jun}.
\begin{lemma} \label{lem:append}
    Let $d>1$, $\mu, \nu\in P_\infty(\rrr^d)$,
    \begin{enumerate}
        \item For $\theta\in \mathbb{S}^{d-1}$, $R^\theta_\sharp(\mu)\in P_\infty(\rrr)$, moreover if $\textup{spt}(\mu)\subset \overline{B_r(0)}$, then $\textup{spt}(R^\theta_\sharp)\subset \overline{[-r, r]}.$
        \item For any sequence $(\theta_i)_{i\in \nnn}\subset \sss^{d-1}$ with $\theta_i\rightarrow \theta$, $\lim_{i\rightarrow+\infty}W_\infty(R^{\theta_i}_\sharp \mu, R^\theta_\sharp \mu) =0$.
        \item The function defined by $\theta\mapsto W_\infty(R^\theta_\sharp\mu, R^\theta_\sharp \nu)$ is continuous. Moreover, if $\textup{spt}(\mu), \textup{spt}(\nu)\subset B_r(0)$ then this is $2r$-Lipschitz continuous.
    \end{enumerate}
\end{lemma}
\begin{proof}
    For part (1), it is clear since if $\textup{spt}(\mu)\subset \overline{B_r(0)}$ then we know that $\mu(\{x: |\theta\cdot x|>r\}) \leq \mu{\overline{B_r(0)}^c} =0$.

For part (2), consider the transport map $\gamma_i\in \Pi(R^{\theta_i}_\sharp \mu, R^\theta_\sharp\mu)$ defined by \begin{equation*}
    \gamma_{i}(A, B) := (\Delta_\sharp\mu)\left[\Big(R^{\theta_i}\times R^\theta\Big)^{-1}(A\times B)\right].
\end{equation*}
    Recall that $\Delta$ is the diagnol map $\Delta(x) = (x, x)$. One can also check that 
    \begin{align*}
        \sup\{|x-y|:(x, y)\in \textup{spt}(\gamma_i)\}&=\sup\{|R^{\theta_i}x-R^\theta y|:(x, y)\in \textup{spt}(\Delta_\sharp \mu)\}\\
        &=\sup\{|R^{\theta_i}x-R^\theta x|:x\in \textup{spt}(\mu)\}\\
        &\leq \sup\{|R^{\theta_i}x-R^\theta x|:||x||\leq r\}\\
        &\leq r|\theta_i-\theta|.
    \end{align*}
The last inequality comes from the fact $R^y x= y\cdot x$ and the Schwarz inequality. Thus, we conclude that $W_\infty(R^{\theta_i}_\sharp \mu, R^\theta_\sharp\mu)\leq r|\theta_i-\theta|\underset{i\rightarrow +\infty}{\rightarrow} 0$.

For part (3), we conclude as in \cite{jun} and note that the triangle inequality on $(W_\infty, \rrr)$ along with (2) applied to both $\mu$ and $\nu$ gives us that continuity of $\theta\mapsto W_\infty(R^\theta_\sharp \mu, R^\theta_\sharp \nu)$. If they are both supported in the ball of radius $r$ then we can see that \begin{align*}
    |W_\infty(R^\theta_\sharp \mu,& R^\theta_\sharp \nu) - W_\infty(R^\phi\sharp \mu, R^\phi\sharp \nu)|\\&\leq |W_\infty(R^\theta_\sharp \mu, R^\theta_\sharp \nu) - W_\infty(R^\phi_\sharp \mu, R^\theta_\sharp \nu)| +|W_\infty(R^\phi_\sharp \mu, R^\theta_\sharp \nu) - W_\infty(R^\phi_\sharp \mu, R^\phi_\sharp \nu)|\\
    &\leq |W_\infty(R^\theta_\sharp \mu, R^\phi_\sharp \mu)| +|W_\infty(R^\theta_\sharp \nu, R^\phi_\sharp \nu)|\\
    &\leq 2r|\theta -\phi|.
\end{align*}
\end{proof}

We can also easily show an upper-bound on the $\infty, q$ sliced Wasserstein distance ($SW_{\infty, q}$) by the $\infty$ Wasserstein distances ($W_\infty$). Recall that Theorem \ref{iff} proves that a Lipschitz-type lower-bound is impossible so the upper bound is all that we can show.

\begin{lemma}\label{lem:Winfineq}
    Given $\mu, \nu \in \mathcal{P}_\infty(\rrr^d)$, and $1\leq q\leq +\infty$, $SW_{\infty, q}(\mu, \nu) \leq W_{\infty}(\mu, \nu)$.
\end{lemma}
\begin{proof}
    Let $\gamma(\mu, \nu)$ be an optimal transport plan so that $$\sup\{|x-y|: (x, y)\in \textup{spt}(\gamma)\} = W_\infty(\mu, \nu).$$ Then we can define $\gamma_\theta \in P_\infty(\rrr\times \rrr)$ by \begin{equation}
        \gamma_\theta(A, B) = \gamma\Big[(R^\theta\times R^\theta)^{-1}(A\times B)\Big].
    \end{equation}
    We can then check that clearly \begin{align*}
        W_\infty(R^\theta_\sharp \mu, R^\theta_\sharp \nu)&\leq \sup\{|x-y|: (x, y)\in \textup{spt}(\gamma_\theta)\}\\
        &=\sup\{|R^\theta x-R^\theta y|: (x, y)\in \textup{spt}(\gamma)\}\\
        &\leq \sup\{|x-y|: (x, y)\in \textup{spt}(\gamma)\}
        \\&= W_\infty(\mu, \nu).
    \end{align*}
    This proves, since $\theta$ was arbitrary, that $SW_{\infty, \infty}(\mu, \nu)\leq W_\infty(\mu, \nu)$. We can then use standard $L^p$ space theory to note that on a space with measure $1$ the $L^p$ norm is bounded by the $L^\infty$ norm for all $1\leq q< +\infty$. As $W_{\infty, q}$ is the $L^q$ norm of $\theta\mapsto W_\infty(R^\theta_\sharp \mu, R^\theta_\sharp \nu)$ we have shown $SW_{\infty, q}(\mu, \nu)\leq SW_{\infty, \infty}(\mu, \nu)\leq W_\infty(\mu, \nu)$.
\end{proof}

\begin{lemma} \label{lem:Cauchyimplybdd}
    If $\{\mu_n\}$ is a Cauchy sequence in $SW_{\infty, 1}$ then there is an $R>0$ such that $\textup{spt}\mu_n\subset B_R(0) $ for all $n\in \nnn$.
\end{lemma}

\begin{proof}
    We will instead prove the contra-positive, and by passing to a subsequence, we may suppose that there is a sequence $\{R_i\}_{i\in \nnn}$ such that $R_{i+1}> 2R_i+3\geq 2^{i+1}$ and a subsequence $\mu_{n_i}$ such that 
    \begin{equation*}
        \textup{spt}\mu_{n_i} \subset B_{R_{i+1}}(0)\hspace{1em}\textup{ and }\hspace{1em}\textup{spt}\mu_{n_i} \not\subset B_{R_{i+1}-2}(0)\supseteq B_{2R_i}(0).
    \end{equation*} 
    Thus, there is a vector $\vec{x}\in \textup{spt}({\mu}_{n_{i+1}})\cap B_{R_{i+1}-1}^c$. Then define $\bar{\theta} := \frac{\vec{x}}{\|\vec{x}\|}$ we the spherical cap around this point \[C_{\bar{\theta}}:=\left\{\theta\in \sss^{d-1} \theta\cdot \bar{\theta}> \frac{1}{2}\right\}.\] 
    We have that for all $\varphi\in C_{\bar{\theta}}$,
    \begin{equation*}
        R^{\varphi}(\vec{x})> \frac{1}{2}(R_{i+1}-1)\geq\frac{1}{2}(2R_i+2)\geq R_i+1\geq\sup\{x\in \textup{spt}{R^\varphi}_\sharp(\mu_{i})\} +1.
    \end{equation*} 
    Which immediately implies \begin{equation*}
        W^\infty(R^{\varphi}_\sharp (\mu_{n_{i+1}}), R^{\varphi}_\sharp(\mu_{n_i}))\geq 1.
    \end{equation*} This implies that for all $i\in \nnn$ we have that \[SW_{\infty, 1}(\mu_{n_{i+1}}, \mu_{n_i})>\int_{C_{\bar{\theta}}} 1 d\mathcal{H}^{d-1} =\frac{d+1}{2^{d}}\]
    This implies that the subsequence, $\{\mu_{n_i}\}_{i\in \nnn}$, is not Cauchy and thus $\{\mu_n\}_{n\in \nnn}$ is also not Cauchy completing the proof by contra-positive.
    
\end{proof}

\begin{lemma}\label{lem:ptwisetounifconv}
    Fix $R>0$ and $d\in \nnn$, Suppose $\{\mu_n\}_{n\in \nnn}\cup\{ \mu\}\subset P_\infty(\rrr^d)$ is $\textup{spt}(\mu_n)\subset B_R(0)$ for all $n\in \nnn$ and $\textup{spt}(\mu)\subset B_R(0)$. Further suppose that $R^\theta_\sharp (\mu_n)\rightarrow \mu$ in $W_\infty$, then $\mu_n\rightarrow\mu$ in $SW_{\infty, \infty}$
\end{lemma}

\begin{proof}
    Fix $\epsilon>0$. We remark that there is a finite mesh $\{\theta_i\}_{i\in I}$ depending on $d$, $R$, and $\epsilon$ such that $\cup_{i\in I} B_{\epsilon/{8R}}(\theta_i)\supset \sss^{d-1}$. We can then note that there is an $N\in \nnn$ such that for all $n\geq N$ and $i\in I$ we have that $W_\infty(R^{\theta_i}_\sharp(\mu_n), R^{\theta_i}_\sharp(\mu))< \epsilon/2$, this is because $I$ is a finite set and our assumption of pointwise convergence. We can then conclude from the facts that $\theta\mapsto R^\theta_{\sharp}(\mu)$ and $\theta\mapsto R^\theta_{\sharp}(\mu_n)$ are $2R$-Lipschitz and by the fineness of our mesh $\{\theta_i\}_{i\in I}$, that for all $\theta\in \sss^{d-1}$ we have that $W_\infty(R^{\theta}_\sharp(\mu_n), R^{\theta}_\sharp(\mu))< \epsilon$ for all $n\geq N$. This then proves directly for all $n\geq N$ we have that $SW_{\infty, \infty}(\mu_n, \mu)< \epsilon$ which concludes the proof.

\end{proof}

\begin{proposition}
        For all $1\leq q\leq +\infty$, $(P_\infty(\rrr^d), W_{\infty, q})$ is a complete separable metric space.
\end{proposition}
\begin{proof}
    First we will show that $SW_{\infty, q}$ is a metric. We can note $SW_{\infty, q}$ is an $L^q$ norm and thus non-negative. Furthermore, we can see that for $\nu_1, \nu_2\in \mathcal{P}_\infty(\rrr)$ that $W_{\infty}(\mu, \nu) \geq W_{p}(\nu_1, \nu_2)$. As such, we can note that if $\mu_1\neq \mu_2\in \mathcal{P}_\infty(\rrr^d)$, then \begin{equation*}
        SW_{\infty, q}(\mu_1, \mu_2)\geq SW_{p, q}(\mu_1, \mu_2) >0.
    \end{equation*}
    Which proves positivity. Furthermore, because $W_\infty(\nu_1, \nu_1)=0$ we know that $SW_{\infty, q}(\mu_1, \mu_1) =0$ since it is the $L^q$ norm of the zero function. As $W_\infty(\cdot, \cdot)$ is symmetric then so is $SW_{\infty, q}$. 
    
    The only metric property left to show is the triangle inequality. Fix $\mu_1, \mu_2,\mu_3\in \mathcal{P}_\infty(\rrr^d)$. Then we can see that by the triangle inequality on $W_\infty$ we know that for all $\theta\in \sss^{d-1}$ that $W_\infty(R^\theta_\sharp \mu_1, R^\theta_\sharp \mu_3) \leq W_\infty(R^\theta_\sharp \mu_1, R^\theta_\sharp \mu_2) + W_\infty(R^\theta_\sharp \mu_2, R^\theta_\sharp \mu_3)$. As such, we know that \begin{align*}
        SW_{\infty, q}(\mu_1, \mu_3) &= ||W_\infty(R^\theta_\sharp \mu_1, R^\theta_\sharp \mu_3)||_{L^q(d\theta)}\\
        &\leq ||W_\infty(R^\theta_\sharp \mu_1, R^\theta_\sharp \mu_2) + W_\infty(R^\theta_\sharp \mu_2, R^\theta_\sharp \mu_3||_{L^q(d\theta)}\\
        &\leq ||W_\infty(R^\theta_\sharp \mu_1, R^\theta_\sharp \mu_2)||_{L^q(d\theta)}  + ||W_\infty(R^\theta_\sharp \mu_2, R^\theta_\sharp \mu_3||_{L^q(d\theta)} \\
        &=SW_{\infty, q}(\mu_1, \mu_2)+SW_{\infty, q}(\mu_2, \mu_3).
    \end{align*}
Thus, we know that $SW_{\infty, q}$ is a metric.
    
    We have thus shown $(\mathcal{P}_\infty(\rrr^d), SW_{q,\infty})$ is a metric space, we now show it is seperable. We can recall that $(P_\infty, W_\infty)$ is a separable metric space (for example 3.2 in \cite{sant}). As such, by if $(\mu_i)_{i\in \nnn}$ is dense in $(P_\infty, W_\infty)$, then it is also dense in $(P_\infty, SW_{\infty, q})$ by Lemma \ref{lem:Winfineq}.

    We now show that the space is complete. Fix $\{\mu_i\}_{i\in \nnn}\subset P_\infty(\rrr^d)$ to be a Cauchy sequence with respect to $SW_{\infty, q}$. Since, for all $\mu, \nu\in \mathcal{P}_\infty(\rrr^d)$ we have that $SW_{\infty, 1}(\mu, \nu)\geq SW_{\infty, q}(\mu, \nu) $, we know that $\{\mu_i\}_{i\in \nnn}$ is Cauchy in $SW_{\infty, 1}$ and thus by Lemma \ref{lem:Cauchyimplybdd} there is an $R$ such that $\textup{spt}(\mu_{i})\subset \overline{B_R(0)}$ for all $i\in \nnn$. We can then recall that since $W_\infty(\nu_1, \nu_2)\geq W_1(\nu_1, \nu_2)$ for all $\nu_1, \nu_2\in P_\infty(\rrr)$ we know that $\mu_i$ is Cauchy in $SW_{1, q}$ and thus there is a $\mu$ such that $\mu_i\rightarrow \mu$ with respect to $SW_{1, q}$ by \cite{jun}. We then have that $\mu_i\rightarrow \mu$ in $W_1$ by the topological equivalence proved in Main Theorem (2) of \cite{jun}. Thus we have that $\textup{spt}(\mu)\subset \overline{B_R(0)}$ being a $W_1$ limit of measures supported on $\overline{B_{R}(0)}$. 
    Moreover, as $W_1(R^\theta_\sharp(\eta), R^\theta_\sharp(\nu))\leq W_1(\eta, \nu)$ we have that for each $\theta$, $R_\sharp^\theta(\mu_i)\rightarrow R^\theta_\sharp(\mu)$ in $W_1$. We now must show this convergence is actually in $W_\infty$.

    To see this note that by assumption we have that $\{\mu_i\}$ is a Cauchy sequence in $SW_{\infty, q}$, and as $\theta\rightarrow R^\theta_\sharp(\mu_i)$ is $2R$-Lipschitz in $W_\infty$ we have that for every $\theta\in \sss^{d-1}$ that $\{R^\theta_\sharp(\mu_i)\}_{\theta\in \sss^d}$ is Cauchy in $W_\infty$. Thus, there is a unique $\mu^\theta\in \mathcal{P}_\infty(\rrr)$ such that $R^\theta_\sharp(\mu_i)\rightarrow \mu^\theta$ in $W_\infty$. Convergence in $W_\infty$ implies and thus also in $W_1$ which implies that $\mu^\theta =R^\theta_\sharp \mu$. Thus we have that there is a $\mu\in P_\infty$ such that $\mu_i\rightarrow \mu$ in $W_{q, \infty}$ by Lemma \ref{lem:ptwisetounifconv} and thus we have shown completeness.
\end{proof}

 We know from \ref{iff} that $W_\infty$ and $SW_{\infty, q}$ are not topologically equivalent. Nonetheless we can recover the topological equivalence across choices of $q$ from the $p<\infty$ case.

\begin{proposition}
    For every  $1\leq q_1, q_2\leq +\infty$and $d\in \nnn$, $(SW_{\infty, q_1}, \mathcal{P}_\infty(\rrr^d))$ and $(SW_{\infty, q_2}, \mathcal{P}_\infty(\rrr^d))$ are topologically equivalent.
\end{proposition}
     
\begin{proof}
    It suffices to consider the case when $q_2=+\infty$. In this case it is clear that if $\mu_i\rightarrow \mu$ with respect to $SW_{\infty, \infty}$ that then we have convergence with respect to $SW_{\infty, q_1}$ as we know that $SW_{\infty, q_1}(\mu_i, \mu)\leq SW_{\infty, \infty}(\mu_i, \mu)$.

    Now suppose, for the sake of contradiction that that $\mu_i\rightarrow \mu$ in $SW_{\infty, q_1}$, but $SW_{\infty, \infty}(\mu_i, \mu)\geq \delta>0$ for all $i\in \nnn$. We can note that there is some $\theta_i$ such that $W_\infty(R^{\theta_i}_\sharp(\mu_i), R^{\theta_i}_\sharp(\mu))\geq \delta$. We can also use Lemma \ref{lem:Cauchyimplybdd} that there is some $R$ such that $\textup{spt}(\mu_{i}), \textup{spt}(\mu)\subset B_R(0)$ for all $i \in \nnn$.
    For all $\theta \in B_{\epsilon}(\theta_i)\cap\sss^{d-1}$, then we have $W^\infty(R^{{\theta}}_\sharp \mu_i, R^{{\theta}}_\sharp \mu)\ge W^\infty(R^{\bar{\theta}}_\sharp \mu_i, R^{\bar{\theta}}_\sharp \mu)-2R\epsilon$ letting $\epsilon = \frac{\delta}{4R}$ we get that $$SW_{\infty, {q_1}}(\mu_i, \mu)\geq SW_{\infty, {1}}(\mu_i, \mu)\geq \frac{\mathcal{H}^{d-1}(C_\epsilon(\theta_i))\delta}{2\mathcal{H}^{d-1}(\del B_1(0))}>0.$$ Which is contradicts the fact that $\mu_i\rightarrow \mu$ in $SW_{\infty, q_1}$ (we are using that the $\mathcal{H}^{d-1}$ measure of $C_\epsilon(\theta_i)$ is independent of $\theta_i\in \sss^{d-1}$).

\end{proof}

\end{document}